\documentclass[UKenglish]{amsart}
\usepackage[utf8]{inputenc}
\usepackage{tikz-cd} 
\usepackage{paralist}
\usepackage{amsmath,amsthm,amssymb,amsfonts, mathtools}
\usepackage{mathrsfs}
\usepackage{xcolor}%
\usepackage{textcomp}%
\usepackage{manyfoot}%
\usepackage{booktabs}%
\usepackage{algorithm}%
\usepackage{hyperref}
\hypersetup{colorlinks,%
	citecolor=black,%
	filecolor=black,%
	linkcolor=black,%
	urlcolor=green,%
	pdftex}

\DeclareMathOperator{\Supp}{Supp}
\DeclareMathOperator{\Diff}{Diff}

\DeclareMathOperator{\coeff}{coeff}
\DeclareMathOperator{\Sing}{Sing}
\DeclareMathOperator{\Spec}{Spec}

\newcommand{\Q}{$\mathbb{Q}$}

\newtheorem{teo}{Theorem}
\numberwithin{teo}{section}
\newtheorem{prop}[teo]{Proposition}
\newtheorem*{claim}{Claim}

\newtheorem{coro}[teo]{Corollary}
\theoremstyle{definition}%
\newtheorem{rmk}[teo]{Remark}%
\newtheorem{de}[teo]{Definition}%
\newtheorem{deprop}[teo]{Def./Prop.}
\author[J. Galindo Jim\'enez]{Jose Galindo Jim\'enez}
\address{Departament de Matem\`atiques, Universitat de Val\`encia, Avinguda Vicent Andr\'es Estell\'es 19, Burjassot, 46100, l'Horta Nord, Pa\'is Valenci\`a, Spain}
\email{jose.galindo(at)uv.es}

\title[Kulikov-Persson-Pinkham theorem via smoothing of dlt models]{Kulikov-Persson-Pinkham theorem via smoothing of dlt models}
\date{}
\thanks{The author has received funding from Basque Government through the BERC 2022-2025 program and BCAM Severo Ochoa 2023-2027 accreditation funded by MICIU/AEI/10.13039/501100011033.  A Catalan translation of this article can be found in
	\href{https://www.uv.es/gajijo/kulikov.catala.pdf}{this link.}
	}
\begin{document}
	
	\maketitle
	\begin{abstract}
		We give an alternative proof of the Kulikov-Persson-Pinkham Theorem for a K\"ahler degeneration of K-trivial smooth surfaces. After running the Minimal Model Program, the obtained minimal dlt model has mild singularities which we resolve via Brieskorn's simultaneous resolutions and toric resolutions.
	\end{abstract}
\section{Introduction}
In the late 1970s, Kulikov obtained a breakthrough in the study of degenerations and singularity theory. Given a semistable degeneration of $K3$ surfaces, he was able to construct a birational semistable minimal model.
\begin{teo}[Kulikov {\cite[Theorem I]{kulikov}}, Persson-Pinkham \cite{PP}]\label{KPP}
	Let $f:\mathcal{Y}\rightarrow \mathbb{D}$ be a semistable\footnote{	It is well known that by Mumford's Semistable Reduction Theorem (cf. \cite[Ch. 4]{mumford}) any flat proper degeneration of complex analytic spaces or schemes over characteristic zero fields can be arranged into a semistable model after finite surjective base change and birational modifications. Moreover, it may be constructed to be projective if the original one was projective.} morphism to a complex disk such that
	\begin{enumerate}
		\item $\mathcal{Y}_t$ is a smooth $K$-trivial surface for $t\in \mathbb{D}, t\neq0$ and
		\item all irreducible components of the special fiber $\mathcal{Y}_0=f^{-1}(0)$ are algebraic.
	\end{enumerate}
	Then there is a birationally equivalent semistable degeneration $f':\mathcal{X}\rightarrow \mathbb{D}$ such that $K_{\mathcal{X}}\sim_{f'}\mathcal{O}_{\mathcal{X}}$.
\end{teo}
More generally in the literature, a semistable degeneration $f:\mathcal{X}\rightarrow C$ with general fibers $\mathcal{X}_t$, such that $K_{\mathcal{X}_t}\sim_{\mathbb{Q}}\mathcal{O}_{\mathcal{X}_t}$, is called a \textit{Kulikov model} if $K_{\mathcal{X}}\sim_{\mathbb{Q},f}\mathcal{O}_{\mathcal{X}}$. An essential feature of Kulikov models is that they are classified based upon the dimension of the dual complex of the central fiber or, equivalently, the action of the monodromy acting on $H^2(\mathcal{X}_t)$ (cf. \cite[Theorem II]{kulikov} and \cite[Proposition 3.3.1]{persson}). Making use of this, they were studied extensively in the previous century (cf. \cite{persson} or \cite{friedman} among others). Moreover, their study influenced later developments in birational geometry, in particular the birth of the Minimal Model Program (MMP).\par
The original proof for $K3$ surfaces by Kulikov was relatively obscure\footnote{At the time of publishing, some were left wondering whether it was valid (see the review of his article in Mathematical Reviews 58, \# 2208).}. Following his ideas, Persson and Pinkham generalized the result with a clearer proof. The strategy relied on combinatorial arguments and complicated subtle complex-analytic (often non-algebraic) constructions such as \textit{generic contractions} and \textit{quasi-degenerations}. To avoid these we turn to the MMP.\par 
Fujino proved that for degenerations of $K$-trivial varieties of arbitrary dimension, the MMP runs, cf. \cite[Theorem 1.1]{fujino}. From a semistable model one obtains a \textit{minimal dlt model}\footnote{Model in the sense of the MMP, i.e., birationally equivalent variety with nef canonical divisor in the same class of singularities.}, essentially the weaker MMP notion of a semistable minimal model. Such a model has been further studied and has very interesting applications, for example in recent developments for the SYZ conjecture (cf. \cite{kollarxu} and \cite{nicaise} among others).\par
In particular, Koll\'ar, Laza, Sacc\`a and Voisin showed that a minimal dlt degeneration of hyper-K\"ahler manifolds serves as an n-dimensional analogue to a Kulikov model. They share many properties, most interestingly regarding the action of the monodromy and topology of the associated dual complexes, cf. \cite[Theorem 0.11]{KLSV}. \par 
This posed the question: Can one obtain honest Kulikov models from these dlt model? The complex analytic case was studied by Iskovkikh and Shokurov, cf. \cite[\S 7]{shokurov}, and the algebraic space case by Odaka. The first ones showed that one may run the MMP for a semistable 3-fold with fibers with non-negative Kodaira dimension and stay in the class of the so-called \textit{semistable singularities}. From this they conclude that the model is smooth with only Du with at most Du Val singularities in the interior of components. Then they smooth them taking the analytic \Q-factorialization. \begin{rmk}
	Semistable singularities are a special case of 3-fold singularities. They have good properties and admit a very explicit description (cf. \cite[\S 1]{shokurovsing}). Nonetheless, they do not appear in the common surveys on MMP (cf. \cite{kollar}, \cite{KM} or \cite{bchm}). As a consequence, they remain unbeknownst to many.
\end{rmk}
Opposed to this, we studied the minimal model using the standard description and properties of dlt singularities. These are defined in a more general setting and have been studied extensively. From this analysis, we recover Theorem \ref{KPP}.
\begin{teo}\label{complex}
	Let $f:\mathcal{Y}\rightarrow C$ be a proper flat dlt surjective morphism from a complex analytic space to $C$ a smooth curve. Assume that
	\begin{enumerate}
		\item  $\mathcal{Y}$ is (globally) $\mathbb{Q}$-factorial and K\"ahler, and
		\item  the general fibers of $f$, $\mathcal{Y}_t$ over closed points $t\in C$, are smooth surfaces with $K_{\mathcal{Y}_t}\sim \mathcal{O}_{\mathcal{Y}_t}$ outside of a finite set $C^0$.
	\end{enumerate}
	Then there exists a finite surjective base change $\pi:C'\rightarrow C$ and a birational map $p$ fitting a commutative diagram
	\[\begin{tikzcd}
		{\mathcal{X}} & {} & {\mathcal{Y}\times_CC'} \\
		& {C'}
		\arrow[dashed, "p", from=1-1, to=1-3]
		\arrow["{f'}"', from=1-1, to=2-2]
		\arrow["\pi_2", from=1-3, to=2-2]
	\end{tikzcd}\]
	such that $f':\mathcal{X}\rightarrow C'$ is a Kulikov model and $p$ is an isomorphism over $\pi^{-1}(C\backslash C^0)$. Here $\pi_2$ denotes the natural projection to $C'$.
\end{teo}
\begin{rmk}
	Note that the hypotheses imply that the model has semistable singularities (compare with \cite[Definition 7.2]{shokurov}).
\end{rmk}
The proof of Theorem \ref{complex} is structured in 3 steps:
\begin{itemize}
	\item [\textbf{Step 1} ] Start with a dlt/semistable model of degenerating $K$-trivial surfaces and as in \cite[Theorem 1.1]{fujino} obtain a realtively $K$-trivial dlt model via the MMP.
	\item [\textbf{Step 2} ] Prove that this minimal dlt model is snc away from the intersection loci of special fibers. More specifically, the total space has at most terminal singularities above isolated canonical surface singularities in the interior of irreducible components of special fibers. (cf. Proposition \ref{1and2})
	\item [\textbf{Step 3} ] As canonical surface singularities are Du Val singularities, they can be resolved \textit{crepantly}, after finite surjective base change, using Brieskorn's resolutions and using toric resolutions, via an argument from Friedman.
\end{itemize}
\begin{rmk}\label{odaka}
	\begin{enumerate}[i.]
		\item In the Appendix to his unpublished article \cite{odaka}, Odaka gave a proof for algebraic spaces under slightly more general hypothesis, allowing non $\mathbb{Q}$-factorial singularities, using a similar strategy. Further he obtained a stronger result as in Step 3  he showed one may perform the simultaneous resolution without the base change, cf. \cite[Lemma A.5]{odaka}. For complex analytic spaces he observed that it should be possible to argue similarly though it is not clear whether the references he cites for this imply it directly, \cite[Lemma A.6] {odaka}.
		\item Step 2 was essentially proven in more generality by Nicaise, Xu and Yue in \cite[Theorem 4.5]{nicaise}. They showed that a (good) minimal dlt model $(\mathcal{X},D)$ of arbitrary dimension with $K_{\mathcal{X}}+D$ Cartier is snc along the 1-dimensional lc strata. In our exposition, the geometry of 3-dimensional singularities allows for a more straightforward argument, see Remark \ref{nicaise}.
	\end{enumerate}
\end{rmk}
\textbf{Acknowledgments.}
	The author is very grateful to Philip Engel and Evgeny Shinder for all their help and guidance. He is also thankful to Vyacheslav Shokurov  and Yuji Odaka for their helpful comments. Furthermore, he thanks Mattias Jonsson and Hyunsuk Kim for pointing out minor mistakes in the first preprint.
\section{Notation and preliminary results}\label{notation}
\subsection{Basic definitions of the Minimal Model Program}
We use the usual definitions and notation from \cite{KM} and \cite{kollar}.\par \smallskip
A \textit{pair} $(X,\Delta)$ consists of a normal complex analytic variety $X$ and a \Q-(Weil) divisor $\Delta $ on it. We consider pairs $(X,\Delta)$ over a smooth variety $S$ satisfying the conditions
\begin{enumerate}
	\item $X$ is normal proper variety that has a dualizing sheaf $\omega_{X/S}$ and
	\item let $\Delta=\sum a_i D_i$ be the irreducible decomposition then no $D_i$ has its support contained in $\Sing(X)$. 
\end{enumerate}
For a birational morphism $f:Y\rightarrow X$, with exceptional divisors $\{E_i\}_{i\in I}$ if $K_X+\Delta $ is \Q-Cartier one can write
\begin{equation}\label{canonicalbundleformula}K_{Y}+f^{-1}_*\Delta\sim_{\mathbb{Q}} f^*(K_X+\Delta )+\sum_{E_i \text{ exceptional}}a(X,\Delta, E_i)E_i,\end{equation}
where $a(X,E_i,\Delta)$ is called the \emph{discrepancy} of $E_i$.  Moreover, set $a(Y, \Delta, D)=-\coeff_D\Delta$ for non-exceptional divisors $D\subset Y$.\par 

Let $(X,\Delta)$ be a pair, for any birational morphism $f:Y\rightarrow X$ and any irreducible divisor $E\subset Y$ one says that\vspace{0.3 cm}\begin{center}
	$(X,\Delta )$ is $\begin{cases}
		
		\text{\textit{terminal}}\\
		\text{\textit{canonical}}\\
		\text{\textit{plt} or \textit{log terminal}}\\
		\text{\textit{log canonical} (lc)}\\
	\end{cases}$ if $a(X,\Delta, E)\begin{cases}
		>0\text{ for all } E \text{ exceptional},\\
		\geq 0\text{ for all } E \text{ exceptional},\\
		>-1\text{ for all } E \text{ exceptional},\\
		\geq-1\text{ for all } E.
	\end{cases}$\vspace{0.2cm}
\end{center}
If $\Delta=0$, then we will say that $X$ is terminal (resp. canonical, etc.) if $(X,0)$ is terminal (resp. canonical, etc). Locally, a point $(p\in X, \Delta)$ is terminal (resp. canonical, etc.) if there is an open neighborhood $p\in U$ with $(U, \Delta|_{U})$ terminal (resp. canonical, etc.). \par
\begin{enumerate}[i.]
	\item Let $(X,\Delta)$ be an lc pair, then an irreducible subvariety $Z\subset X$ is a \textit{log canonical center} or \textit{lc center} if there is a birational morphism $f:Y\rightarrow X$ and a divisor $E\subset Y$ such that $a(X,\Delta, E)=-1$ and $\overline{f(E)}^X=Z$. Here $\overline{(\cdot)}^X$ denotes the scheme-theoretic closure. More generally, for any divisor $E\subset Y$ such that $\overline{f(E)}^X\subseteq Z$ one says that $E$ is a \textit{divisor lying over $Z$}.
	\item We say that $(X,\Delta)$ is \textit{simple normal crossings} or \textit{snc} if $X$ is smooth and $\Delta$ has simple normal crossing support. The largest open set $X^{snc}\subset X$ such that $(X^{snc}, \Delta|_{X^{snc}})$ is an snc pair is called the \textit{simple normal crossing locus} or \textit{snc locus}. A flat proper morphism $f:X\rightarrow C$ is called \textit{semistable} if $(X,X_t)$ is snc for all $t$ closed point in $C$, where $X_t=f^{-1}(t)$ with reduced structure.
	\item A log canonical pair is called \textit{divisorial log terminal} or \textit{dlt} if none of the lc centers of $(X,\Delta)$ lie over $X\backslash X^{snc}$. Analogously, a morphism $f:X\rightarrow C$ is called \textit{dlt} if $(X,X_t)$ is dlt for all $t$ closed point in $C$.
	\item A variety $X$ is called (globally) \Q\textit{-factorial} if every Weil divisor on $X$ is \Q-Cartier. \item 
	A pair $f:(X,\Delta)\rightarrow S$ is called \textit{minimal} if $K_X+\Delta$ is $f$-nef, i.e., nef for every closed irreducible curve contained in a closed fiber of $f$.
	\item A \textit{birational map} $f:(X,\Delta_X)\dashrightarrow (Y,\Delta_Y)$ will be called crepant if $K_X+\Delta_X\sim f^*(K_Y+\Delta_Y)$ and $\Delta_Y=f_*\Delta_X$.
\end{enumerate}

\subsection{Dlt pairs and the different.}
We now recall some useful facts about dlt pairs and the different. In general, it will be of interest to have control of dlt pairs when $\Delta$ is reduced. For this reason, we draw attention to the following theorem.
\begin{teo}\cite[Theorem 4.16]{kollar}\label{dltpairs}
	Let $(X,\Delta)$ be a dlt pair and $V_1,...,V_r$ the irreducible divisors that appear in $\Delta$ with coefficient 1.
	\begin{enumerate}
		\item The $k$-codimensional lc centers of $(X,\Delta)$ are exactly the irreducible components of the various $V_{i_1}\cap...\cap V_{i_k}$.
		\item Every irreducible component  of $V_{i_1}\cap...\cap V_{i_k}$ is normal of pure codimension $k$.
		\item Let $Z\subset X$ be any lc center. Assume that $V_i$ is (\Q-)Cartier for some $i$ and $Z\nsubseteq V_i$. Then every irreducible component of $V_i|_Z$ is also (\Q-)Cartier.
	\end{enumerate}
\end{teo}
\begin{de}
	Let $(X,\Delta)$ be a pair and $V_i$, $i\in I$ the prime components of $\Delta$. A \textit{stratum} of $\Delta$ is a connected component of the  intersection $\Delta_J=\cap_{j\in J}V_j$ for some non-empty $J\subset I$. 
\end{de}
In particular, in the setting of the previous theorem, if $\Delta$ is reduced and effective, the irreducible components of each strata correspond to log canonical centers.\par
Working with divisors, it is always useful to have some form of adjunction formula, but singularities pose obstructions to it. Luckily, dlt pairs allow for a ``fix" of adjunction by adding a correcting term (also holds in a more general setting, see \cite[\S 4.1]{kollar}).
\begin{deprop}\cite[Rmk. 8.2]{fujinodlt}
	Let $(X, \Delta)$ be a dlt pair and $V$
	an irreducible component appearing with coefficient 1 in $\Delta$. Then there exists a unique \Q-divisor $\Diff_V(\Delta-V)$ on $V$, defined by the equation
	\[(K_X + \Delta)|_V= K_V + \Diff_V(\Delta-V),\]
	known as the \textit{different}. Moreover, $(V, \Diff_V(\Delta-V))$ is a dlt pair.
\end{deprop}
In particular, notice the different ``inherits'' the property of dlt-ness. Computing the different is generally not easy, but it can be under good hypotheses. One deduces easily from Theorem \ref{dltpairs} and \cite[Prop. 9.2]{fujinodlt} the following description of the different.
\begin{coro}\cite[Paragraphs 6 and 15]{kollarxu}\label{different}
	Let $(X,\Delta)$ be a dlt pair such that $\Delta$ is effective and reduced with irreducible components $V_i$ and $K_X+\Delta$ is (\Q-)Cartier. Let $D_i$ be the sum of codimension 2 strata of $\Delta$ supported on $V_i$. Then \[\Diff_{V_i}(\Delta-V_i)=D_i,\] and $D_i$ is (\Q-)Cartier. Consequently, the pair $(V_i, D_i)$ is dlt and satisfies the adjunction formula
	\[(K_X + \Delta)|_{V_i} = K_{V_i} + D_i.\]
\end{coro}

\section{Proof of Kulikov-Persson-Pinkham Theorem.}\label{complexproof}
We want to follow the strategy laid out in the introduction. Step 1 is just running the MMP, to obtain a terminal and \Q-factorial analytic space $\mathcal{X}$ with $(\mathcal{X},\mathcal{X}_t)$ being dlt. Step 2 consists of studying the geometry of dlt pairs from the previous section to obtain an accurate description of its singularities. 
\begin{prop}\label{onlyduval}
	Let $X$ be 3-dimensional and $(X,\Delta)$ a dlt pair such that $\Delta$ is effective and reduced with irreducible components $V_i$ being \Q-Cartier\footnote{In the context of dlt morphisms and degenerations, this is sometimes referred to as \textit{vertically \Q-Cartier} (cf. \cite[Definition 1.1]{boucksom}) or \textit{good} dlt model (cf. \cite[\S 1.11]{nicaise}).}. Suppose that $K_{X}+\Delta$ is Cartier, then the $V_i$ have at most isolated canonical singularities away from the 1-dimensional strata.
\end{prop}
\begin{proof}
	Let $D_i$ be the sum of 1-dimensional strata supported on $V_i$. By Theorem \ref{dltpairs}, all $V_i$ are normal, hence have at most isolated singularities, and so are the irreducible components of $D_i$, so these are smooth curves.\par 
	By Corollary \ref{different} $(V_i, D_i)$ is dlt and $K_{V_i}+D_i$ is Cartier. By definition of dlt pair, there exists a closed subset $Z:=V_i\backslash V_i^{snc}$ such that any divisor over $Z$ has discrepancy $>-1$. Since $K_{V_i}+D_i$ is Cartier, by formula (\ref{canonicalbundleformula}), discrepancies are integers and thus the discrepancy must be $\geq 0$.\par
	$Z$ cannot contain neither the 0-dimensional strata of $D_i$ nor an irreducible curve supported on $D_i$, since these are log canonical centers by Theorem \ref{dltpairs}. Suppose there is a closed point $p\in \Supp D_i\cap Z$. Since the discrepancy is $\geq 0$, it is canonical, hence regular in $V_i$ by \cite[Theorem 2.29]{kollar}. Thus, $\Supp D_i$ is in the smooth locus of $V_i$. Away from it, one has at most canonical singularities.
\end{proof}
\begin{coro}\label{step2}
	Let $(X,\Delta)$ be a dlt pair such that $X$ is 3-dimensional and Gorenstein and $\Delta$ is effective and reduced with irreducible components $V_i$ being \Q-Cartier. Suppose $K_X+\Delta$ is Cartier and $X\backslash\Supp \Delta\subset X^{snc}$. Then \[X\backslash X^{snc}=\cup_i\Sing(V_i),\] where $\Sing(V_i)$ consists only of isolated canonical singularities in the interior of $V_i$ and $X$ is terminal. In particular, $(X,\Delta)$ is snc in an open neighborhood of the 1-dimensional strata.
\end{coro}
\begin{proof}
	By the previous proposition, it suffices to prove that the $V_i$ are Cartier, as in this case, a non-regular point of $X$ will also be a non-regular point of $V_i$.\par
	We claim that $X$ is terminal. Since $(X,\Delta)$ is dlt, any divisor $E$ not lying over the snc locus has $a(E,X,\Delta)>-1$. However, $K_X+\Delta$ is Cartier and thus, as in the previous proposition, $a(E,X,\Delta)\geq0$. Then using that $X\backslash\Supp \Delta\subset X^{snc}$ by \cite[Lemma 2.27]{KM} $a(E,X,0)>0$.\par
	Since $X$ is 3-dimensional, terminal and Gorenstein, a lemma from Kawamata (cf. \cite[Lemma 5.1]{factorial}) implies that every \Q-Cartier divisor on $X$ is Cartier. Hence the $V_i$ are Cartier and $\Sing(X)\subset \cup_i \Sing(V_i)$ and thus $X\backslash X^{snc}=\cup_i \Sing(V_i)$.
\end{proof}
\begin{rmk}\label{nicaise}
	Nicaise, Xu and Yue proved the last assertion of this corollary (which is the main point) in \cite[Theorem 4.5]{nicaise} for arbitrary dimension. The proof differs in how to show that the $V_i$ are Cartier around the 1-dimensional strata. They do this with a sophisticated argument, but in the 3-dimensional setting, it is more straightforward to apply Kawamata's Lemma.
\end{rmk}
\begin{prop}\label{1and2}
	Let $f:\mathcal{Y}\rightarrow C$ be a proper flat dlt surjective morphism from a complex analytic space to $C$, a smooth curve. Assume that
	\begin{enumerate}
		\item $\mathcal{Y}$ is $\mathbb{Q}$-factorial and K\"ahler and,
		\item the general fibers of $f$, $\mathcal{Y}_t$ over closed points $t\in C$, are smooth surfaces with $K_{\mathcal{Y}_t}\sim \mathcal{O}_{\mathcal{Y}_t}$ outside of a finite set $C^0$.
	\end{enumerate}
	Then there exists a birational morphism $p:\mathcal{Y}\rightarrow \mathcal{X}$ fitting a commutative diagram
	\[\begin{tikzcd}
		{\mathcal{Y}} && {\mathcal{X}} \\
		& {C}
		\arrow["{p}", from=1-1, to=1-3]
		\arrow["{f}"', from=1-1, to=2-2]
		\arrow["{f'}", from=1-3, to=2-2]
	\end{tikzcd}\]
	such that $p$ is an isomorphism over $C\backslash C^0$ and 
	\begin{enumerate}
		\item $\mathcal{X}$ is terminal, \Q-factorial and $K_{\mathcal{X}}\sim_f \mathcal{O}_{\mathcal{X}}$.
		\item $(\mathcal{X},\mathcal{X}_t)$ is snc in an open neighborhood of the 1-dim. strata of $\mathcal{X}_t$, for all $t\in C$ closed point.
		\item The irreducible components $V_{i,t}$ of $\mathcal{X}_t$ have at most isolated canonical singularities and $$\mathcal{X}\backslash\mathcal{X}^{snc}= \cup_{i,t}\Sing(V_{i,t}).$$
	\end{enumerate}
\end{prop} 
\begin{proof}
	\textbf{Step 1.} The MMP runs \cite[Theorem 1.1]{horing} and one obtains $f':\mathcal{X}\rightarrow C$ a dlt morphism with $\mathcal{X}$ \Q-factorial and $K_{\mathcal{X}}$ being $f'$-nef. As in \cite[Theorem 1.1]{fujino}, one has that $K_{\mathcal{Y}_t}\sim \mathcal{O}_{\mathcal{Y}_t}$ for a general fiber implies that $K_{\mathcal{X}}\sim_{f'} \mathcal{O}_{\mathcal{X}}$.\par
	For any closed fiber $\mathcal{X}_t$, by adjunction (cf. \cite[\S 4.1]{kollar}) one has
	\begin{equation*}
		\mathcal{O}_{\mathcal{X}_t}\sim (K_{\mathcal{X}}+\mathcal{X}_t)|_{\mathcal{X}_t}=K_{\mathcal{X}_t}.
	\end{equation*}
	Moreover, a general fiber $\mathcal{Y}_t$ is $K$-trivial so by adjunction an arbitrary divisor from $K_{\mathcal{Y}}$ induces a principal divisor on $\mathcal{Y}_t$ and so $|K_{\mathcal{Y}}|$ consists of a linear combination divisors on the special fibers. So as $K_{\mathcal{Y}}$ intersects trivially with any curve of a general fiber $\mathcal{Y}_t$, they cannot be contracted by any step of the MMP. Hence, $p$ is an isomorphism over $C\backslash C^0$.\par
	\textbf{Step 2.} $\mathcal{X}$ is relatively $K$-trivial, so both $K_{\mathcal{X}}$ and $K_{\mathcal{X}}+\mathcal{X}_t$ are Cartier. Moreover, every divisor is \Q-Cartier by \Q-factoriality. Thus away from other singular fibers in some (Zariski) neighborhood $\mathcal{U}\subset\mathcal{X}$ of $\mathcal{X}_t$ the  pair $(\mathcal{U}, \mathcal{X}_t)$ satisfies the assumptions of Corollary \ref{step2} which yields the desired result.
\end{proof}
From here, one deduces that the only obstruction to obtaining an honest Kulikov model after running the MMP is the canonical singularities of the special fibers. These are exactly Du Val singularities, cf. \cite[Theorem 4.5]{KM} and admit a simultaneous resolution after finite surjective base change. However, the base change produces toric singularities along the intersection loci, using a trick from Friedman these may be resolved crepantly.
\begin{proof}[Proof of Theorem \ref{complex}]
	After applying Proposition \ref{1and2}, it suffices to show we can resolve crepantly all the isolated Du Val singularities on the minimal dlt model.
	\par \textbf{Step 3.} Without loss of generality, assume there is a single (possibly) singular fiber which we call $\mathcal{Y}_0$. Pick a Du Val singularity $x\in V_i$ in an irreducible component of $\mathcal{Y}_0$. There is a neighborhood $x\in \mathcal{U}\subseteq \mathcal{Y}$ for which there are only Du Val singularities. Hence by \cite{brieskorn} there exists a finite ramified surjective base change $\pi_0:\tilde{C}\rightarrow C$ giving a Cartesian square
	\[\begin{tikzcd}
		{\mathcal{Z}} & \mathcal{Y} \\
		{\tilde{C}} & C
		\arrow["{g}", from=1-1, to=1-2]
		\arrow["{h}"', from=1-1, to=2-1]
		\arrow["\lrcorner"{anchor=center, pos=0.125}, draw=none, from=1-1, to=2-2]
		\arrow["f_0", from=1-2, to=2-2]
		\arrow["{\pi_0}"', from=2-1, to=2-2]
	\end{tikzcd}\]
	that admits a simultaneous resolution $q$
	\[\begin{tikzcd}
		{\tilde{\mathcal{Z}}} & {\mathcal{V}\subset \mathcal{Z}} & x\in \mathcal{U}\subset \mathcal{Y} \\
		{\tilde{C}} & {\tilde{C}} & C.		
		\arrow["{q}", from=1-1, to=1-2]
		\arrow["{\tilde{h}}"', from=1-1, to=2-1]
		\arrow["{g}", from=1-2, to=1-3]
		\arrow["{h}"', from=1-2, to=2-2]
		\arrow["\lrcorner"{anchor=center, pos=0.125}, draw=none, from=1-2, to=2-3]
		\arrow["f_0", from=1-3, to=2-3]
		\arrow["{=}"', from=2-1, to=2-2]
		\arrow["{\pi_0}"', from=2-2, to=2-3]
	\end{tikzcd}\]
	Notice that $\mathcal{Z}$ is $h$-trivial as the relative canonical divisor does not change for a finite surjective base change. Moreover, $q$ gives a minimal resolution fiberwise, therefore $q$ is a birational morphism and an isomorphism in codimension 1, thus $q^*K_{\mathcal{Z}}\sim K_{\tilde{\mathcal{Z}}}$. Globally we obtain
	\[\begin{tikzcd}
		{\tilde{\mathcal{Z}}} & {\mathcal{Z}} & \mathcal{Y} \\
		{\tilde{C}} & {\tilde{C}} & C
		\arrow["{q}", from=1-1, to=1-2]
		\arrow["{\tilde{h}}"', from=1-1, to=2-1]
		\arrow["{g}", from=1-2, to=1-3]
		\arrow["{h}"', from=1-2, to=2-2]
		\arrow["\lrcorner"{anchor=center, pos=0.125}, draw=none, from=1-2, to=2-3]
		\arrow["f_0", from=1-3, to=2-3]
		\arrow["{=}"', from=2-1, to=2-2]
		\arrow["{\pi_0}"', from=2-2, to=2-3]
	\end{tikzcd}.\]
	Now $\tilde{\mathcal{Z}}$ may have singularities arising from the effect of the base change in the lc locus.
	\begin{claim}
		There exists a crepant birational morphism over $\tilde{C}$
		\[\begin{tikzcd}
			{\mathcal{Z}'} && {\tilde{\mathcal{Z}}} \\
			& {\tilde{C}}
			\arrow["{j}", from=1-1, to=1-3]
			\arrow["{h'}"', from=1-1, to=2-2]
			\arrow["{\tilde{h}}", from=1-3, to=2-2]
		\end{tikzcd}\]
		such that it is an isomorphism outside of a neighborhood $\mathcal{V}_0$ of the intersection locus of $\tilde{\mathcal{Z}_0}$; it is a log resolution of singularities over $\mathcal{V}_0$ and $h':\mathcal{Z}'\rightarrow \tilde{C}$ is a flat, proper morphism. 
	\end{claim}
	\begin{proof}[Proof of Claim.]
		The claim follows by applying the arguments from Friedman in the proof of \cite[Proposition 1.2]{friedmanbase}. An order $n:1$ base change creates toric singularities, which locally analytically around the triple points look like
		\[\Spec \mathbb{C}[x,y,z,t]/(xyz-t^n)\]
		and the double curves become curves of $A_{n-1}$ singularities. These are toric and may be resolved by the standard toric resolution, i.e., subdividing the corresponding fans. Most importantly, it turns out that in this setting the process is crepant. This is done locally around the intersection loci so away from the canonical singularities, so it is exactly as in the snc case described in the reference.
	\end{proof}
	After applying the claim, $(\mathcal{Z}',(\mathcal{Z}_0'))$ satisfies being snc away from the Du Val singularities, has one Du Val singularity less than $(\mathcal{Y},\mathcal{Y}_0)$ and maintains $K$-triviality. Repeating this process to $(\mathcal{Z}',\mathcal{Z}_0')$, leads inductively to a finite surjective morphism $\pi:C'\rightarrow C$, a birational map and a flat proper morphism $f':\mathcal{X}\rightarrow C'$ fitting a commutative diagram
	\[\begin{tikzcd}
		{\mathcal{X}} & {} & {\mathcal{Y}\times_CC'} \\
		& {C'}
		\arrow[dashed, "p", from=1-1, to=1-3]
		\arrow["{f'}"', from=1-1, to=2-2]
		\arrow["", from=1-3, to=2-2]
	\end{tikzcd}.\]
	Any terminal singularity $(x\in \mathcal{X})$ must lie over the interior of the irreducible components $V_i$ since $\mathcal{X}_0$ is a Cartier divisor. This does not occur since we resolved all singularities in the interior of the $V_i$. Hence $\mathcal{X}$ is smooth, $K_{\mathcal{X}}\sim_{f'}\mathcal{O}_{\mathcal{X}}$ and $\mathcal{X}_0$ is an snc divisor q.e.d.
\end{proof}
\begin{coro}[Kulikov-Persson-Pinkham Theorem]
	Let $f:\mathcal{X}\rightarrow \mathbb{D} $ be a Kähler proper degeneration of smooth $K$-trivial surfaces over the complex disk. Then there is a finite and surjective base change $\pi:\mathbb{D}\rightarrow \mathbb{D}$ and $\mathcal{X}'$ a smooth manifold fitting a commutative diagram
	\[\begin{tikzcd}
		{\mathcal{X}'} & {} & {\mathcal{X}\times_\mathbb{D}\mathbb{D}} \\
		& {\mathbb{D}}
		\arrow[dashed, "p", from=1-1, to=1-3]
		\arrow["{f'}"', from=1-1, to=2-2]
		\arrow["\pi_2", from=1-3, to=2-2]
	\end{tikzcd}.\]
	Where $p$ is a birational map which is an isomorphism outside the central fiber and $f':\mathcal{X}'\rightarrow \mathbb{D}$ is a semistable degeneration with $K_{\mathcal{X}'}\sim_{f'}\mathcal{O}_{\mathcal{X}'}$.
\end{coro}

\begin{rmk}\label{projectiveassumption}\cite[Remark A.6]{odaka}
	For Theorem \ref{KPP}, originally Persson and Pinkham did not assume the variety to be K\"ahler just that the fibers were algebraic.	This is assumed to run the MMP for complex analytic spaces as it has not been established in the setting of Theorem \ref{KPP}.\par
Remark that there is an obstruction to generalising this. A degeneration of $K3$ surfaces that does not admit a Kulikov model was constructed by Nishiguchi in \cite[\S 4]{counterexample}. This degeneration satisfies that the central fiber contains a non-algebraic complex analytic surface as an irreducible component (so-called Kodaira class VII surfaces).
\end{rmk}

\begin{rmk}[Surface case $\kappa(\mathcal{X}_t)=0$]\label{numerically}
	Kulikov models do not exist in general for smooth surfaces with \textit{numerically} trivial canonical divisor. A counterexample for degenerations of Enriques surfaces is given by the ``flower pots'' in \cite[3.3 and Appendix 2]{persson}. By \cite[Theorem 1.2]{fujino} running the MMP gives a minimal dlt model $f:\mathcal{X}\rightarrow C$ with $K_{\mathcal{X}}\sim_{\mathbb{Q},f}\mathcal{O}_{\mathcal{X}}$.\par
	What would fail applying the same strategy is that in Step 2 it is essential that $K_{\mathcal{X}}$ is Cartier to deduce that: Discrepancies of $(V_i, D_i)$ are integers and that irreducible components $V_i$ are Cartier.
\end{rmk}
	
	\bibliographystyle{alpha}
	\bibliography{referencies.bib} 

\begin{thebibliography}{KKMSD73}

\bibitem[Art74]{artin}
M.~Artin.
\newblock Algebraic construction of {Brieskorn}'s resolutions.
\newblock {\em J. Algebra}, 29:330--348, 1974.

\bibitem[BCHM10]{bchm}
Caucher Birkar, Paolo Cascini, Christopher~D. Hacon, and James McKernan.
\newblock Existence of minimal models for varieties of log general type.
\newblock {\em J. Am. Math. Soc.}, 23(2):405--468, 2010.

\bibitem[BFJ16]{boucksom}
S{\'e}bastien Boucksom, Charles Favre, and Mattias Jonsson.
\newblock Singular semipositive metrics in non-{Archimedean} geometry.
\newblock {\em J. Algebr. Geom.}, 25(1):77--139, 2016.

\bibitem[Bri71]{brieskorn}
Egbert Brieskorn.
\newblock Singular elements of semi-simple algebraic groups.
\newblock In {\em Proc. Int. Cong. Math., Nice}. Gauthier-Villars, 1971.

\bibitem[FM83]{friedman}
Robert Friedman and David~R. Morrison.
\newblock {\em The Birational Geometry of Degenerations}, volume~29 of {\em
  Progress in mathematics}.
\newblock Birkh{\"a}user, 1983.

\bibitem[Fri83]{friedmanbase}
Robert Friedman.
\newblock Base change, automorphisms, and stable reduction for type {III} {K}3
  surfaces.
\newblock {\em The birational geometry of degenerations (Cambridge, Mass.,
  1981)}, 29:277--298, 1983.

\bibitem[{F}uj07]{fujinodlt}
{O}samu {F}ujino.
\newblock {W}hat is log terminal?
\newblock {\em Flips for 3-folds and 4-folds}, 35:49--62, 2007.

\bibitem[Fuj11]{fujino}
Osamu Fujino.
\newblock Semi-stable minimal model program for varieties with trivial
  canonical divisor.
\newblock {\em Proc. Japan Acad. Ser. A Math. Sci.}, 87(3):25--30, 2011.

\bibitem[HP16]{horing}
Andreas H{\"o}ring and Thomas Peternell.
\newblock Minimal models for {K}{\"a}hler threefolds.
\newblock {\em Invent. Math.}, 203(1):217--264, 2016.

\bibitem[IS05]{shokurov}
Vasily~A. Iskovskikh and Vyacheslav~V. Shokurov.
\newblock Birational models and flips.
\newblock {\em Russ. Math. Surv.}, 60(1):27--94, 2005.

\bibitem[Kaw88]{factorial}
Yujiro Kawamata.
\newblock Crepant blowing-up of 3-dimensional canonical singularities and its
  application to degenerations of surfaces.
\newblock {\em Ann. Math. (2)}, 127(1):93--163, 1988.

\bibitem[Kaw94]{positive}
Yujiro Kawamata.
\newblock Semistable minimal models of threefolds in positive or mixed
  characteristic.
\newblock {\em J. Algebr. Geom.}, 3(3):463, 1994.

\bibitem[KKMSD73]{mumford}
G.~Kempf, F.~Knudsen, D.~Mumford, and Bernard Saint-Donat.
\newblock {\em Toroidal embeddings. {I}}, volume 339 of {\em Lect. Notes Math.}
\newblock Springer, 1973.

\bibitem[KLSV18]{KLSV}
J{\'a}nos Koll{\'a}r, Radu Laza, Giulia Sacc{\`a}, and Claire Voisin.
\newblock Remarks on degenerations of hyper-{K{\"a}hler} manifolds.
\newblock {\em Ann. Inst. Fourier}, 68(7):2837--2882, 2018.

\bibitem[KM98]{KM}
János Kollár and Shigefumi Mori.
\newblock {\em Birational Geometry of Algebraic Varieties}, volume 134 of {\em
  Camb. Tracts Math.}
\newblock Cambridge University Press, 1998.

\bibitem[Kol98]{realkollarI}
J{\'a}nos Koll{\'a}r.
\newblock Real algebraic threefolds. {I}: {Terminal} singularities.
\newblock {\em Collect. Math.}, 49(2-3):335--360, 1998.

\bibitem[Kol99]{kollarreal}
J{\'a}nos Koll{\'a}r.
\newblock Real algebraic threefolds. {II}: {Minimal} model program.
\newblock {\em J. Am. Math. Soc.}, 12(1):33--83, 1999.

\bibitem[Kol13]{kollar}
J{\'a}nos Koll{\'a}r.
\newblock {\em Singularities of the minimal model program.}, volume 200 of {\em
  Camb. Stud. Adv. Math.}
\newblock Cambridge University Press, 2013.

\bibitem[Kul77]{kulikov}
Viktor Kulikov.
\newblock Degenerations of {K}3 surfaces and {E}nriques surfaces.
\newblock {\em Math. USSR, Izv.}, 11(5):957, 1977.

\bibitem[KX16]{kollarxu}
J{\'a}nos Koll{\'a}r and Chenyang Xu.
\newblock The dual complex of {Calabi}-{Yau} pairs.
\newblock {\em Invent. Math.}, 205(3):527--557, 2016.

\bibitem[Mil68]{milnor}
John~W. Milnor.
\newblock {\em Singular points of complex hypersurfaces}, volume~61 of {\em
  Ann. Math. Stud.}
\newblock Princeton University Press, Princeton, NJ, 1968.

\bibitem[Mor88]{moriflip}
Shigefumi Mori.
\newblock Flip theorem and the existence of minimal models for 3-folds.
\newblock {\em J. Am. Math. Soc.}, 1(1):117--253, 1988.

\bibitem[Nis88]{counterexample}
Kenji Nishiguchi.
\newblock Degeneration of {K}3 surfaces.
\newblock {\em J. Math. Kyoto Univ.}, 28(2):267--300, 1988.

\bibitem[NXY19]{nicaise}
Johannes Nicaise, Chenyang Xu, and Tony~Yue Yu.
\newblock The non-archimedean {SYZ}fibration.
\newblock {\em Compos. Math.}, 155(5):953--972, 2019.

\bibitem[Oda21]{odaka}
Yuji Odaka.
\newblock On log minimality of weak {K}-moduli compactifications of
  {Calabi}-{Yau} varieties.
\newblock Preprint, {arXiv}:2108.03832 [math.{AG}] (2021), 2021.

\bibitem[Per77]{persson}
Ulf Persson.
\newblock {\em On degenerations of algebraic surfaces}, volume 189 of {\em Mem.
  Am. Math. Soc.}
\newblock American Mathematical Society, 1977.

\bibitem[PP81]{PP}
Ulf Persson and Henry Pinkham.
\newblock Degeneration of surfaces with trivial canonical bundle.
\newblock {\em Ann. Math. (2)}, 113:45--66, 1981.

\bibitem[Sho93]{shokurovsing}
Vyacheslav~V. Shokurov.
\newblock Semistable 3-fold flips.
\newblock {\em Russ. Acad. Sci., Izv., Math.}, 42(2):371--425, 1993.

\end{thebibliography}
	
\end{document}